\documentclass[12pt]{amsart}

\usepackage{tgpagella}
\usepackage{euler}
\usepackage[T1]{fontenc}
\usepackage{amsmath, amssymb}
\usepackage[hidelinks]{hyperref}
\usepackage[english]{babel}
\usepackage{mathrsfs}
\usepackage{eucal}
\usepackage[all]{xy}
\usepackage{tikz}

\newtheorem{thm}{Theorem}[section]
\newtheorem*{thm*}{Theorem}
\newtheorem{lem}[thm]{Lemma}
\newtheorem{fact}[thm]{Fact}

\newtheorem{prop}[thm]{Proposition}
\newtheorem*{prop*}{Proposition}

\newtheorem{cor}[thm]{Corollary}
\newtheorem*{cor*}{Corollary}

\theoremstyle{definition}
\newtheorem{defn}[thm]{Definition}
\newtheorem*{defn*}{Definition}

\newtheorem{remark}[thm]{Remark}

\newtheorem{question}[thm]{Question}

\newtheorem*{question*}{Question}
\newtheorem*{Pquestion*}{Popa's question}

\newtheorem*{conv*}{Convention}

\def\bb{\mathbb}

\def\bb{\mathbb}

\def\cal{\mathcal}

\def\u{\mathsf 1}

\newcommand{\cstar}{$\mathrm{C}^*$}

\makeatletter

\def\dotminussym#1#2{%
  \setbox0=\hbox{$\m@th#1-$}%
  \kern.5\wd0%
  \hbox to 0pt{\hss\hbox{$\m@th#1-$}\hss}%
  \raise.6\ht0\hbox to 0pt{\hss$\m@th#1.$\hss}%
  \kern.5\wd0}

\def \u{\mathcal U}
\def \val{\operatorname{val}}

\begin{document}


\title{On Tsirelson pairs of \cstar-algebras}
\author{Isaac Goldbring and Bradd Hart}

\address{Department of Mathematics\\University of California, Irvine, 340 Rowland Hall (Bldg.\# 400),
Irvine, CA 92697-3875}
\email{isaac@math.uci.edu}
\urladdr{http://www.math.uci.edu/~isaac}

\address{Department of Mathematics and Statistics, McMaster University, 1280 Main St., Hamilton ON, Canada L8S 4K1}
\email{hartb@mcmaster.ca}
\urladdr{http://ms.mcmaster.ca/~bradd/}

\thanks{I. Goldbring was partially supported by NSF grant DMS-2504477. B. Hart was partially supported by the NSERC}

\begin{abstract}
We introduce the notion of a Tsirelson pair of \cstar-algebras, which is a pair of \cstar-algebras for which the space of quantum strategies obtained by using states on the minimal tensor product of the pair is dense in the space of quantum strategies obtained by using states on the maximal tensor product.  We exhibit a number of examples of such pairs that are ``nontrivial'' in the sense that the minimal tensor product and the maximal tensor product of the pair are not isomorphic.  For example, we prove that any pair containing a \cstar-algebra with Kirchberg's QWEP property is a Tsirelson pair.  We then introduce the notion of a \cstar-algebra with the Tsirelson property (TP) and establish a number of closure properties for this class.  We also show that the class of \cstar-algebras with the TP forms an elementary class (in the sense of model theory), but that this class does not admit an ``effective'' axiomatization.
\end{abstract}
\maketitle

\section{Introduction}

\textbf{Tsirelson's problem} in quantum information theory (see \cite{tsirelson,tsirelson2}) asks whether or not, for any given pair of natural numbers $k,n\in \bb N$ with $k,n\geq 2$, the set $C_{qc}(k,n)$ of \textbf{quantum commuting correlations} coincides with the set $C_{qa}(k,n)$ of \textbf{quantum asymptotic correlations}.  The former set consists of correlation matrices $p(a,b|x,y)\in [0,1]^{k^2n^2}$ (suggestively written as conditional probabilities for their connections with nonlocal games) for which there exists a Hilbert space $\cal H$, POVMS $A^x=(A^x_1,\ldots,A^x_n)$ and $B^y=(B^y_1,\ldots,B^y_n)$ on $\cal H$ for all $x,y\in [k]:=\{1,\ldots,k\}$ satisfying $[A^x_a,B^y_b]:=A^x_aB^y_b-B^y_bA^x_a=0$ for all $x,y\in [k]$ and $a,b\in [n]$, and a unit vector $\xi\in \cal H$ for which $p(a,b|x,y)=\langle A^x_aB^y_b \xi,\xi\rangle$.  (Precise definitions of POVMs and any other undefined notions will be given later in the paper.)  The latter set consists of the subset of those quantum commuting strategies that can be approximated by quantum commuting strategies for which the Hilbert space $\cal H$ can be decomposed as $\cal H_A\otimes \cal H_B$ and for which the POVMS $A^x$ (respectively $B^y$) act only on $\cal H_A$ (respectively only on $\cal H_B$).  It is a consequence of the landmark result $\operatorname{MIP}^*=\operatorname{RE}$ \cite{MIP*} that there are pairs $(k,n)$ for which $C_{qa}(k,n)$ is a proper subset of $C_{qc}(k,n)$, thus providing a negative answer to Tsirelson's problem in general.

Work of Fritz \cite{Fr} and independently Junge et. al. \cite{Junge} allows one to recast Tsirelson's problem in terms of states on tensor products of certain \cstar-algebras.  (It is this connection that allows one to conclude a negative solution to Kirchberg's QWEP problem-and thus, ultimately, a negative solution to the Connes Embedding Problem-from a negative solution to Tsirelson's problem.)  Indeed, set $\bb F(k,n)$ to be the group freely generated by $k$ elements of order $n$ and let $C^*(\bb F(k,n))$ denote its universal group \cstar-algebra.  It was established that, for any $p(a,b|x,y)\in [0,1]^{k^2n^2}$, one has:
\begin{itemize}
    \item $p\in C_{qa}(k,n)$ if and only if there are POVMs $A^x$ and $B^y$ in $C^*(\bb F(k,n))$ and a state $\phi$ on $C^*(\bb F(k,n))\otimes C^*(\bb F(k,n))$ such that $$p(a,b|x,y)=\phi(A^x_a\otimes B^y_b);$$
    \item $p\in C_{qc}(k,n)$ if and only if there are POVMs $A^x$ and $B^y$ in $C^*(\bb F(k,n))$ and a state $\phi$ on $C^*(\bb F(k,n))\otimes_{\max} C^*(\bb F(k,n))$ such that $$p(a,b|x,y)=\phi(A^x_a\otimes B^y_b).$$
\end{itemize}
Throughout this paper, $\otimes$ always denotes the minimal tensor product of \cstar-algebras.

The preceding fact suggests looking at correlations of the form $\phi(A^x_a\otimes B^y_b)$, where $A^x$ and $B^y$ are POVMs from some fixed pair $(C,D)$ of \cstar-algebras and $\phi$ is a state either on $C\otimes D$ or $C\otimes_{\max}D$; in fact, we will consider the closures of these two sets of correlations, which we will denote by $C_{\min}(C,D,k,n)$ and $C_{\max}(C,D,k,n)$ respectively.  (In general, we see no reason why these sets of correlations should be closed, although they do happen to be so for the case that $C=D=C^*(\bb F(k,n))$.)  It is standard fare that $C_{\min}(C,D,k,n)\subseteq C_{\max}(C,D,k,n)$, $C_{\min}(C,D,k,n)\subseteq C_{qa}(k,n)$, and $C_{\max}(C,D,k,n)\subseteq C_{qc}(k,n)$; for the convenience of the reader, proofs of these facts will be given in Section 3.  

Motivated by the above discussion, we will call a pair of \cstar-algebras $(C,D)$ a \textbf{Tsirelson pair} if $C_{\min}(C,D,k,n)=C_{\max}(C,D,k,n)$ for all $(k,n)$.  Denoting by $\bb F_\infty$ the free group on a countably infinite set of generators and using the fact that the canonical projection $C^*(\bb F_\infty)\to C^*(\bb F(k,n))$ admits a ucp lift for each $(k,n)$, the fact that Tsierelson's problem has a negative answer implies that $(C^*(\bb F_\infty),C^*(\bb F_\infty))$ is not a Tsirelson pair.

Following Pisier, a pair of \cstar-algebras $(C,D)$ is called a \textbf{nuclear pair} if $C\otimes D = C\otimes_{\max}D$.  Clearly, any nuclear pair is a Tsirelson pair.  We show, however, that the class of Tsirelson pairs is much larger than the class of nuclear pairs.  For example, we show that if one element of the pair has Kirchberg's QWEP property, then the pair is a Tsirelson pair.  This follows from a more general fact, namely that the class of Tsirelson pairs is closed under ``quotients.''

We then move on to studying \cstar-algebras with the \textbf{Tsirelson property} (TP), which we define to be those \cstar-algebras which form a Tsirelson pair with any other \cstar-algebra.  Continuing with our analogy with nuclear pairs, \cstar-algebras with the TP are the analog of nuclear \cstar-algebras in this context.  By the previous paragraph, the class of \cstar-algebras with the TP includes the class of \cstar-algebras with the QWEP; we leave open the question of whether or not these classes coincide.  We establish a variety of closure properties of the class of \cstar-algebras with the TP and show that they form an elementary class (in the sense of model theory) but that this class does not admit any ``effective'' set of axioms.  As a consequence, we show that the class of \cstar-algebras without the TP is not closed under ultraproducts.

We are grateful to Micha\l\   Banacki, Narutaka Ozawa, Vern Paulsen, Christopher Schafhauser, David Sherman, and Aaron Tikuisis for useful discussions and comments about this work.

\section{Preliminaries}

\subsection{WEP, QWEP, and LP}

Recall that if $C\subseteq D$ are \cstar-algebras, a \textbf{weak conditional expectation} from $D$ onto $C$ is a ucp map $D\to C^{**}$ restricting to the canonical inclusion $C\hookrightarrow C^{**}$.  If there is a weak conditional expectation from $D$ onto $C$, then we say that $C$ is \textbf{relatively weakly injective (r.w.i.)} in $D$.  It is known that $C$ is r.w.i. in $D$ if and only if:  for any  \cstar-algebra $E$, we have that $C\otimes_{\max}E\subseteq D\otimes_{\max}E$, that is, the canonical map $C\otimes_{\max}E\to D\otimes_{\max}E$ is isometric.  $C$ has the \textbf{weak expectation property (WEP)} if it is r.w.i. in any \cstar-algebra containing it.

Recalling the notion of a nuclear pair from the introduction, Kirchberg's fundamental result from \cite{K} on \cstar-algebras with WEP reads as follows:

\begin{fact}\label{Kirchbergfact}
$C$ has the WEP if and only if $(C,C^*(\bb F_\infty))$ is a nuclear pair.
\end{fact}

Recall also that a \cstar-algebra $C$ has the \textbf{QWEP property} if it is a quotient of a \cstar-algebra with the WEP.

The \cstar-algebra $C$ has the \textbf{lifting property (LP)} if:  given any ucp map $\Phi:C\to D/J$, where $J$ is an ideal of some \cstar-algebra $D$, there is a ucp map $\Psi:C\to D$ such that $\Phi=\pi\circ \Psi$, where $\pi:D\to D/J$ is the canonical quotient map.  







We will need the following recent characterization of the LP due to Pisier \cite{Pisier}:

\begin{fact}\label{Pisierfact}
$D$ has the LP if and only if:  for any family $(C_i)_{i\in I}$ of \cstar-algebras and any ultrafilter $\u$ on $I$, the canonical map $(\prod_\u C_i)\otimes_{\max}D\to \prod_\u (C_i\otimes_{\max}D)$ is isometric.
\end{fact}

\subsection{Ultraproduct states}

 Throughout this paper, $\u$ denotes a nonprincipal ultrafilter on some index set $I$.  Given a family of \cstar-algebras $(C_i)_{i\in I}$, we define their \textbf{ultraproduct with respect to $\u$} to be the quotient of the \cstar-algebra $\prod_{i\in I} C_i$ of all uniformly bounded sequences by the ideal of elements $(A_i)_{i\in I}$ for which $\lim_\u \|A_i\|=0$.  It is well-known that $\prod_\u C_i$ is a \cstar-algebra once again.  Given $(A_i)_{i\in I}\in \prod_{i\in I} C_i$, we denote its coset in $\prod_\u C_i$ by $(A_i)_\u$.  When $C_i=C$ for all $i\in I$, we refer to the ultraproduct as the \textbf{ultrapower of $C$ with respect to $\u$} and denote it by $C^\u$.  There is a canonical embedding of $C$ into $C^\u$, called the \textbf{diagonal embedding}, given by mapping $c\in C$ to the coset of the sequence indexed by $I$ with constant value $c$.
 
 If $\phi_i\in S(C_i)$ is a state on $C_i$ for each $i\in I$, we let the \textbf{ultraproduct state} (also known in the literature as the \textbf{limit state}) $\phi:=(\phi_i)_\u$ on $\prod_\u C_i$ be defined by $\phi((A_i)_\u):=\lim_\u \phi_i(A_i)$.
 
We will need the following lemma later in this paper:

\begin{lem}\label{extension}
For any ultraproduct $\prod_\u C_i$ of \cstar-algebras, the set of ultraproduct states is weak*-dense in the set of all states on $\prod_\u C_i$.
\end{lem}

\begin{proof}
Suppose, towards a contradiction, that this is not the case.  Let $X$ denote the weak*-closure of the set of ultraproduct states in $S(\prod_\u C_i)$. 
If $X \neq S(\prod_\u C_i)$, then by the Hahn-Banach Separation Theorem, there is a self-adjoint element $A\in \prod_\u C_i$ with $\|A\|=1$ such that $\psi(A)=0$ for all $\psi\in X$. 
Write $A=(A_i)_\u$ with each $A_i$ a self-adjoint element of $C_i$.  For each $i\in I$, take $\psi_i\in S(C_i)$ such that $\psi_i(A_i)=\|A_i\|$.  Setting $\psi:=(\psi_i)_\u\in X$, we have that $1=\|A\|=\lim_\u \|A_i\|=\psi(A)$, a contradiction.
\end{proof}

\subsection{POVMs in \cstar-algebras}

Given a \cstar-algebra $C$, a \textbf{positive operator-valued measure} or \textbf{POVM} of length $n$ in $C$ is a finite collection $A_1,\ldots,A_n$ of positive elements in $C$ such that $A_1+\cdots+A_n=I$.  

It is well-known that there is a 1-1 correspondence between POVMs in $A$ of length $n$ and ucp maps $\Phi:\bb C^n\to A$, where $A_1,\ldots,A_n$ as above corresponds to the ucp map $\Phi$ given by $\Phi(e_i)=A_i$ for $i=1,\ldots,n$, where $e_1,\ldots,e_n$ is the standard basis for $\bb C^n$.





\begin{lem}\label{liftingPOVMs}
Suppose that $A_1,\ldots,A_n$ is a POVM in the \cstar-algebra $C$ and $\pi:D\to C$ is a surjective *-homomorphism.  Then there is a POVM $B_1,\ldots,B_n$ in $D$ such that $\pi(B_i)=A_i$ for $i=1,\ldots,n$.
\end{lem}

\begin{proof}
Let $\Phi:\bb C^n\to C$ be the ucp map given by $\Phi(e_i)=A_i$ for $i=1,\ldots,n$. Since $\bb C^n$ has the LP, there is a ucp map $\Psi:\bb C^n\to D$ such that $\Phi=\pi\circ \Psi$.  Setting $B_i:=\Psi(e_i)$, we have that $B_1,\ldots,B_n$ is the desired POVM in $D$.
\end{proof}

A special case of the previous lemma is the following:

\begin{prop}\label{definable}
Suppose that $(C_i)_{i\in I}$ is a family of \cstar-algebras and $\u$ is an ultrafilter on $I$.  Let $C:=\prod_\u C_i$.  Then for any POVM $A_1,\ldots,A_n$ of length $n$ from $C$, there are POVMs $A_{1,i},\ldots,A_{n,i}$ of length $n$ in $C_i$ for each $i\in I$ such that $A_a=(A_{a,i})_\u$ for all $a\in [n]$.
\end{prop}

In model theoretic terms, the previous proposition says that the set of POVMs of length $n$ is definable in the theory of \cstar-algebras.
The previous proposition together with a standard ``compactness and contradiction'' argument yield the following corollary.  However, for our purposes in the last section, it behooves us to give a more direct and explicit proof.

\begin{cor}\label{almostnear}
For any $n\geq 1$, $\epsilon>0$, and sequence $A_1,\ldots,A_n$ of positive elements from a \cstar-algebra $C$ for which $\|\sum_{i=1}^n A_i-I\|<\frac{\epsilon}{2}$, there is a POVM $B_1,\ldots,B_n$ in $C$ with $\|A_i-B_i\|<\epsilon$ for all $i=1,\ldots,n$.
\end{cor}

\begin{proof}
First suppose that $\sum^n_{i = 1} A_i \leq I$.  In this case, the desired POVM can be obtained by letting $B_i = A_i$ for $i < n$ and $B_n = A_n + (I - \sum^n_{i = 1} A_i)$.  

Suppose, on the other hand, that $\sum^n_{i = 1} A_i \geq I$ and set $C_i = \frac{1}{1 + \frac{\epsilon}{2}} A_i$. We then have that $\| A_i - C_i\| \leq \frac{\epsilon}{2}$ and, by functional calculus, that $\sum^n_{i = 1} C_i\leq I$.  By the first case, there is a POVM $B_1,\ldots,B_n$ such that $\|B_i-C_i\|<\frac{\epsilon}{2}$ for each $i=1,\ldots,n$, whence $\|A_i-B_i\|<\epsilon$, as desired.
\end{proof} 

\section{Tsirelson pairs of \cstar-algebras}

We now precisely define the correlation sets mentioned in the introduction.

\begin{defn}
Given \cstar-algebras $C$ and $D$, we let $C_{\min}(C,D,k,n)$ (respectively $C_{\max}(C,D,k,n))$ denote the \emph{closure} of the set of correlations of the form \\$\phi(A^x_a\otimes B^y_b)$, where $A^1,\ldots,A^k$ are POVMs of length $n$ from $C$, $B^1,\ldots,B^k$ are POVMs of length $n$ from $D$, and $\phi$ is a state on $C\otimes D$ (respectively a state on $C\otimes_{\max} D$).
\end{defn}

\begin{remark}
In the preceding definition, we took closures as in many arguments, having closed sets of correlations is preferable.  It is not clear to us if taking the closure is necessary (although we suspect that it is in general).
\end{remark}

The following is well-known, but we include a proof for the sake of the reader.  For the proof, we recall that the sets $C_{qa}(k,n)$ and $C_{qc}(k,n)$ are closed, convex subsets of $[0,1]^{k^2n^2}$ for all $(k,n)$.

\begin{lem}\label{basic}
For any pair of \cstar-algebras $C$ and $D$ and any $(k,n)$, we have:
\begin{enumerate}
    \item $C_{\min}(C,D,k,n)\subseteq C_{\max}(C,D,k,n)$.
    \item $C_{\min}(C,D,k,n)\subseteq C_{qa}(k,n)$.
    \item $C_{\max}(k,n)\subseteq C_{qc}(k,n)$.
    \item If $C'$ and $D'$ are \cstar-algebras for which $C\subseteq C'$ and $D\subseteq D'$, then $$C_{\min}(C,D,k,n)\subseteq C_{\min}(C',D',k,n).$$
\end{enumerate}
\end{lem}

\begin{proof}
(1) follows from the fact that any state on $C\otimes D$ induces a state on $C\otimes_{\max}D$ via the canonical surjection $C\otimes_{\max}D\to C\otimes D$.  To see (2), fix faithful representations $C\subseteq B(\cal H)$ and $D\subseteq B(\cal K)$, whence $C\otimes D\subseteq B(\cal H\otimes \cal K)$ is a faithful representation.  Now given any state $\phi$ on $C\otimes D$, POVMs $A^x$ and $B^y$ from $C$ and $D$ respectively, and $\epsilon>0$, there are positive real numbers $\lambda_1,\ldots,\lambda_m$ with $\sum_{j=1}^m\lambda_j=1$ and unit vectors $\psi_1,\ldots,\psi_m\in \cal H\otimes \cal K$ for which $$|\phi(A^x_a\otimes B^y_b)-\sum_{j=1}^m\lambda_j \langle (A^x_a\otimes B^y_b)\psi_j,\psi_j\rangle|<\epsilon$$ for all $x,y\in [k]$ and $a,b\in [n]$ (see, for example, \cite[Proposition B.5]{Fr}).  It remains to use the fact that $C_{qa}$ is a closed, convex set.

(3) is an immediate consequence of the GNS construction while (4) follows from the fact that $C\otimes D\subseteq C'\otimes D'$ and any state on $C\otimes D$ extends to a state on $C'\otimes D'$. 
\end{proof}


\begin{defn}
We say that a pair $(C,D)$ of \cstar-algebras is a \textbf{Tsirelson pair} if ${C_{\min}(C,D,k,n)}={C_{\max}(C,D,k,n)}$ for all $(k,n)$.
\end{defn}

Note that every nuclear pair is a Tsirelson pair.  Our next goal is to give a useful reformulation of the notion of Tsirelson pair.  We first need a preliminary lemma.  Recall that, for $n\in \bb N$, a \cstar-algebra $C$ is called \textbf{$n$-subhomogeneous} if all irreducible representations of $C$ have dimension at most $n$.  $C$ is called \textbf{subhomogeneous} if it is $n$-subhomogeneous for some $n\in \bb N$.  Recall also that subhomogeneous \cstar-algebras are nuclear.

\begin{lem}\label{matrix}
If $C$ is not subhomogeneous, then for every $m\in \bb N$, there is a ucp embedding of $M_m(\bb C)$ in $C$.
\end{lem}

\begin{proof}
Fix $m\in \bb N$.  By assumption, there is an irreducible
representation of $C$ on some Hilbert space $\cal H$ with $\dim(\cal H) > m$. Consider an orthogonal projection $p\in  B(\cal H)$
of rank $m$ and consider the ucp map $T: C \to M_m(\bb C)$ given by $T(c)=pcp$.
Then $T$ restricts to a surjective *-homomorphism on the
multiplicative domain for $T$ by the Kadison Transitivity theorem.
Since $M_m(\bb C)$ has the lifting property, we can find a ucp left inverse
of $T$, establishing the lemma.
\end{proof}

\begin{prop}\label{TPreformulation}
The pair $(C,D)$ is a Tsirelson pair if and only if $C_{\max}(C,D,k,n)\subseteq C_{qa}(k,n)$ for all $(k,n)$.
\end{prop}

\begin{proof}
The forward direction is immediate from Lemma \ref{basic}(2).  For the converse direction, fix $(k,n)$ and suppose that $C_{\max}(C,D,k,n)\subseteq C_{qa}(k,n)$.  If either $C$ or $D$ are subhomogeneous, then $(C,D)$ is a nuclear pair and hence a Tsirelson pair.  Thus we may assume that neither $C$ nor $D$ are subhomogeneous.  In this case, by Lemma \ref{matrix}, for each $m\in \bb N$, there is a u.c.p. embedding $\Phi:M_m(\bb C)\otimes M_m(\bb C)\to C\otimes D$.  Consequently, $C_{qa}(k,n)\subseteq {C_{\min}(C,D,k,n)}$ for all $(k,n)$.  By Lemma \ref{basic}(1) and the assumption, we see that ${C_{\min}(C,D,k,n)}={C_{\max}(C,D,k,n)}=C_{qa}(k,n)$, and thus $(C,D)$ is a Tsirelson pair. 
\end{proof}

The proof of the previous proposition motivates the following definition:

\begin{defn}
The pair $(C,D)$ is said to be a \textbf{strong Tsirelson pair} if $${C_{\min}(C,D,k,n)}={C_{\max}(C,D,k,n)}=C_{qa}(k,n)$$ for all $(k,n)$.
\end{defn}

\begin{cor}\label{strongpair}
Given a pair $(C,D)$ of \cstar-algebras, exactly one of the following occurs:
\begin{itemize}
    \item $(C,D)$ is not a Tsirelson pair.
    \item One of $C$ or $D$ is subhomogeneous, in which case $(C,D)$ is a nuclear pair (and thus a Tsirelson pair), but not a strong Tsirelson pair.
    \item $(C,D)$ is a strong Tsirelson pair.
\end{itemize}
\end{cor}

\begin{proof}
Given the proof of Proposition \ref{TPreformulation}, the only assertion that needs to be established is that if one of $C$ or $D$ is subhomogeneous, then $(C,D)$ is not a strong Tsirelson pair.  
While there is perhaps a more elementary proof of this statement, we use Theorem \ref{noeffective} below and thus defer the proof until after that theorem.
\end{proof}

\begin{remark}
One might wonder about the condition that ${C_{\max}(C,D,k,n)}=C_{qc}(k,n)$ for all $(k,n)$.  This happens, for example, when $C=D=C^*(\bb F_\infty)$.  (This follows from the fact that the canonical map $C^*(\bb F_\infty)\to C^*(\bb F(k,n))$ has a ucp lift together with the characterizations of $C_{qa}(k,n)$ and $C_{qc}(k,n)$ in terms of states on the minimal and maximal tensors products of $C^*(\bb F(k,n))$ with itself.)  However, by the negative solution to Tsirelson's problem, we have that no such pair can be a Tsirelson pair.
\end{remark}

We end this section with a useful closure property of the class of Tsirelson pairs.  To state this, we call a pair $(C',D')$ a \textbf{quotient} of the pair $(C,D)$ if $C'$ is a quotient of $C$ and $D'$ is a quotient of $D$.

\begin{prop}\label{quotientclosure}
The set of Tsirelson pairs is closed under taking quotients.
\end{prop}

\begin{proof}
This follows immediately from Lemma \ref{liftingPOVMs}.
\end{proof}

\section{\cstar-algebras with the Tsirelson property}

\begin{defn}
We say that a \cstar-algebra $C$ has the \textbf{Tsirelson property (TP)} if $(C,D)$ is a Tsirelson pair for any \cstar-algebra $D$.
\end{defn}

Recall that a \cstar-algebra $C$ has WEP if and only if $(C,C^*(\bb F_\infty))$ is a nuclear pair.  The following lemma is the analog for the TP and Tsirelson pairs.

\begin{lem}\label{WTP}
For a \cstar-algebra $C$, the following are equivalent:
\begin{enumerate}
    \item $C$ has the TP.
    \item $(C,D)$ is a Tsirelson pair for every separable \cstar-algebra $D$.
    \item $(C,C^*(\bb F_\infty))$ is a Tsirelson pair.
\end{enumerate}
\end{lem}

\begin{proof}
It is clear that (1) implies (2).  To prove that (2) implies (1), fix $(k,n)$, an arbitrary \cstar-algebra $D$, POVMs $A^x$ and $B^y$ in $C$ and $D$ respectively of length $n$, and a state $\phi$ on $C\otimes_{\max}D$.  Let $D'$ be the subalgebra of $D$ generated by the $B^y_b$'s.  We have a canonical map $C\otimes_{\max}D'\to C\otimes_{\max}D$, which induces a state $\phi'$ on $C\otimes_{\max}D'$.  Thus $\phi(A^x_a\otimes B^y_b)=\phi'(A^x_a\otimes B^y_b)$ and the latter correlation belongs to $C_{qa}(k,n)$ by assumption.

(2) clearly implies (3).  The implication (3) implies (2) follows from Proposition \ref{quotientclosure}.
\end{proof}

We note the following closure properties of the class of \cstar-algebras with the TP:

\begin{prop}\label{closureWTP}
The class of \cstar-algebras with the TP is closed under quotients, r.w.i. subalgebras, and ultraproducts.
\end{prop}

\begin{proof}
Closure under quotients follows from Proposition \ref{quotientclosure}.  Closure under r.w.i. subalgebras follows from the definitions.  We now prove closure under ultraproducts.  Suppose that $(C_i)_{i\in I}$ is a family of \cstar-algebras with the TP and $\u$ is an ultrafilter on $I$.  Set $C:=\prod_\u C_i$.  By Lemma \ref{WTP}, it suffices to show that $(C,C^*(\bb F_\infty))$ is a Tsirelson pair.  Fix $k$ and $n$ and for each $x,y\in [k]$, consider a POVM $A^x$ in $C$ of length $n$.  By Proposition \ref{definable}, for each $x\in [k]$, there is a POVM $A^x_{1,i},\ldots,A^x_{n,i}$ in $C_i$ such that $A^x_a=(A^x_{a,i})_\u$ for all $a\in [n]$.  Fix also POVMs $B^y$ in $C^*(\bb F_\infty)$ and a state $\phi\in S(C\otimes_{\max}C^*(\bb F_\infty))$.  Set $p(a,b|x,y):=\phi(A^x_a\otimes B^y_b)$.  By Fact \ref{Pisierfact}, the canonical map $C\otimes_{\max}C^*(\bb F_\infty)\to \prod_\u (C_i\otimes_{\max}C^*(\bb F_\infty))$ is  isometric, whence we can consider $\phi$ as a state on the latter algebra by extension.  By Lemma \ref{extension}, given $\epsilon>0$, there are states $\phi_i\in S(C_i\otimes_{\max}C^*(\bb F_\infty))$ such that $|\lim_\u \phi_i(A^x_{a,i}\otimes B^y_b))-p(a,b|x,y)|<\epsilon$.  By assumption, the correlations $p_i(a,b|x,y):=\phi_i(A^x_{a,i}\otimes B^y_b)$ belong to $C_{qa}(k,n)$, whence so does $p$ by the previous sentence, letting $\epsilon$ tend to $0$.
\end{proof}

The following corollary yields a large class of \cstar-algebras with the TP:

\begin{cor}\label{QWEPTP}
If $C$ has the QWEP property, then $C$ has the TP.
\end{cor}

\begin{proof}
This follows from the closure of the class of \cstar-algebras with the TP under quotients together with the fact that every \cstar-algebra with the WEP has the TP.  To see this latter statement, suppose that $ C$ has the WEP; it suffices to show that $(C,C^*(\bb F_\infty))$ has the TP, which is indeed the case since this pair is a nuclear pair by Fact \ref{Kirchbergfact}(1).
\end{proof}

\begin{question}\label{QWEPquestion}
Is there an example of a \cstar-algebra with the TP but not the QWEP property?
\end{question}

Given that the class of \cstar-algebras with the QWEP is closed under direct products, one potential approach to Question \ref{QWEPquestion} would be to establish a negative answer to the next question:

\begin{question}
Is the class of \cstar-algebras with the TP closed under direct products?
\end{question}

Since the class of \cstar-algebras with the TP is closed under quotients, if it were closed under direct products, then it would in fact be closed under arbitrary reduced products, that is, quotients of direct products by arbitrary filters.

Just like the QWEP property, the TP is closed under inductive limits:

\begin{prop}\label{inductive}
If $(C_i)_{i\in I}$ is a directed family of \cstar-algebras with the TP, then so is the direct limit $\varinjlim C_i$.
\end{prop}

\begin{proof}
Set $C:=\varinjlim C_i$.  Fix $k$ and $n$ and a \cstar-algebra $D$.  For each $x,y\in [k]$, consider POVMs $A^x$ and $B^y$ of length $n$ in $C$ and $D$ respectively and a state $\phi\in S(C\otimes_{\max}D)$.  By Corollary \ref{almostnear}, given $\epsilon>0$, there is $i\in I$ and POVMs $\bar A^x$ in $C_i$ of length $n$ such that the image of $\bar A^x_a$ in $C$ under the canonical map is within $\epsilon$ of $A^x_a$ for each $x\in [k]$ and $a\in [n]$.  Moreover, the canonical map of $C_i$ into $C$ induces a map $C_i\otimes_{\max}D\to C\otimes_{\max}D$, inducing a state $\phi'$ on $C_i\otimes_{\max}D$.  It follows that the correlation $p'(a,b|x,y):=\phi'(\bar A^x_a\otimes B^y_b)$ is within $\epsilon$ of the correlation $p(a,b|x,y):=\phi(\bar A^x_a\otimes B^y_b)$.  Since $C_i$ has the TP, $p'\in C_{qa}(k,n)$.  Since $\epsilon$ is arbitrary, it follows that $p\in C_{qa}(k,n)$, whence $(C,D)$ is a Tsirelson pair. 
\end{proof}

We list one further closure property of the class of \cstar-algebras with the TP (see \cite[Exercise 2.3.11]{BO} for the analogous property of the class of nuclear \cstar-algebras):

\begin{prop}\label{TPapp}
Suppose that $C$ is a \cstar-algebra with the property that, for any finite subset $F$ of $C$ and any $\epsilon>0$, there is a \cstar-algebra $C'$ with the TP and u.c.p. maps $\psi_1:C\to C'$ and $\psi_2:C'\to C$ such that $\|(\psi_2\circ \psi_1)(c)-c\|<\epsilon$ for all $c\in F$.  Then $C$ also has the TP.
\end{prop}

\begin{proof}
Fix a \cstar-algebra $D$.  Fix also $k$ and $n$ and, for each $x,y\in [k]$, POVMs $A^x$ and $B^y$ in $C$ and $D$ respectively of length $n$.  Fix also a state $\phi\in S(C\otimes_{\max}D)$; we wish to show that $p(a,b|x,y):=\phi(A^x_a\otimes B^y_b)$ belongs to $C_{qa}(k,n)$.  Fix $\epsilon>0$ and take a \cstar-algebra $C'$ with the TP and u.c.p. maps $\psi_1:C\to C'$ and $\psi_2:C'\to C$ with respect to the finite set $F:=\{A^x_a \ : \ x\in [k], a\in [n]\}$ and the given $\epsilon$.  Let $\bar A^x_a:=\psi_1(A^x_a)$ and note that each $\bar A^x$ is a POVM in $C'$.  Let $\phi':=\phi\circ (\psi_2\otimes \operatorname{id}_D)$, a state on $C'\otimes_{\max}D$.  Let $p'(a,b|x,y):=\phi'(\bar A^x_a\otimes B^y_b)$.  By assumption, $p'\in C_{qa}(k,n)$.  Since $\|p-p'\|<\epsilon$ and $\epsilon$ was arbitrary, it follows that $p\in C_{qa}(k,n)$, as desired.
\end{proof}

\begin{remark}
It appears to be an open question as to whether the conclusion of Proposition \ref{TPapp} holds if one replaces ``TP'' by ``QWEP.''  This might perhaps lead to a way to separate these two classes as asked in Question \ref{QWEPquestion}.
\end{remark}

\begin{defn}
A \cstar-algebra $C$ has the \textbf{strong Tsirelson property (STP)} if it has the TP and is not subhomogeneous.
\end{defn}

\begin{lem}
A \cstar-algebra $C$ has the STP if and only if $(C,D)$ is a strong Tsirelson pair for every non-subhomogeneous \cstar-algebra $D$.
\end{lem}

\begin{lem}
The \cstar-algebras with STP are closed under r.w.i. subalgebras and ultraproducts.
\end{lem}

\begin{proof}
This follows from Proposition \ref{WTP} and the fact that the class of \cstar-algebras that are not subhomogeneous is closed under r.w.i. subalgebras and ultraproducts.
\end{proof}

\section{Model-theoretic connections}

\begin{thm}
Both the class of \cstar-algebras with the TP and the class of \cstar-algebras with the STP are elementary.
\end{thm}

\begin{proof}
We first deal with the case of \cstar-algebras with the TP.  We must show that this class is closed under ultraproducts and ultraroots.  The first item was established in Proposition \ref{closureWTP}.  Since any \cstar-algebra is r.w.i. in its ultrapower (see \cite[Proposition 2(2)]{GoldQWEP}), closure under ultraroots also follows from Proposition \ref{closureWTP}.  

Since the class of \cstar-algebras that are not subhomogeneous is axiomatizable (see \cite[Subsection 2.5.5]{Muenster}), the second statement of the theorem follows from the first.
\end{proof}

By Proposition \ref{inductive}, both of the aforementioned classes are inductive.

While one might wish for an ``explicit'' axiomatization of the class of \cstar-algebras with the TP or STP, the next result will show that this cannot be the case.

\begin{thm}\label{noeffective}
There can be no effectively axiomatizable theory $T$ in the language of pairs of \cstar-algebras such that all models of $T$ are Tsirelson pairs and at least one model of $T$ is a strong Tsirelson pair.
\end{thm}

\begin{proof}
Suppose, towards a contradiction, that such a theory $T$ exists.  Work now in the language of pairs of \cstar-algebras expanded by a new unary predicate symbol.  One can then consider the effectively axiomatizable theory $T'$ whose models are of the form $(C,D,P)$, where $(C,D)$ is a model of $T$ and $P(c,d)=\phi(c\otimes d)$ for some state $\phi$ on $C\otimes_{\max}D$.  That this class of structures is axiomatizable (and effectively so) is due to the fact that states on the maximal tensor product are just extensions of unital linear functionals on the algebraic tensor product that are positive on the algebraic tensor product.

One can now reach a contradiction to the main result of \cite{MIP*} as follows.  We assume familiarity with nonlocal games and their quantum and commuting values as defined in \cite{MIP*}.  Given a nonlocal game $\frak G$ with $k$ questions, $n$ answers, probability distribution $\pi$, and decision predicate $D$, one can consider the first-order sentence $\sigma_{\frak G}$ in the language of $T'$ given by
$$\sup_{A}\sup_B \sum_{(x,y)\in [k]}\pi(x,y)\sum_{(a,b)\in [n]}D(x,y,a,b)P(A^x_a,B^y_b).$$ Here, it is understood that $A$ ranges over $k$-tuples of POVMs of length $n$ from the first \cstar-algebra in the pair while $B$ ranges over $k$-tuples of POVMs of length $n$ from the second \cstar-algebra in the pair.  That these quantifications are allowed, that is, that the sets in question are definable, follows from Proposition \ref{definable}.  Since the modulus of definability of these sets is explicit (in fact, by Corollary \ref{almostnear}, it is $\epsilon\mapsto \frac{\epsilon}{2}$), $\sigma_{\frak G}$ is effectively describable from the description of $\frak G$.  (See \cite[Section 2]{universal} for more on these matters.)

The assumptions on the theory $T$ imply that
$$\sup\{\sigma_{\frak G}^{(C,D,P)} \ : \ (C,D,P)\models T'\}=\val^*(\frak G),$$ where $\val^*(\frak G)$ is the quantum entangled value of $\frak G$.  Since $T'$ is effectively axiomatizable, the Completeness Theorem for continuous logic implies that one can effectively enumerate upper bounds for the left-hand side of the above display, whence also for $\val^*(\frak G)$.  Since effective lower bounds for $\val^*(\frak G)$ always exist, this allows one to effectively approximate $\val^*(\frak G)$, which, by the  main result of \cite{MIP*}, would allow one to decide the halting problem, yielding a contradiction.
\end{proof}

\begin{remark}
The assumption of at least one model that is a strong Tsirelson pair is necessary as the class of pairs of abelian \cstar-algebras is clearly effectively axiomatizable and consists entirely of nuclear (and thus Tsirelson) pairs.
\end{remark}

\begin{remark}
We can now finish the proof of Corollary \ref{strongpair}.  Recall that we claimed that if one $C$ or $D$ is subhomogeneous, then $(C,D)$ is not a strong Tsirelson pair.  Suppose, towards, a contradiction that $C$ is $n$-subhomogeneous and $(C,D)$ is a strong Tsirelson pair.  Then letting $T$ denote the theory of pairs of \cstar-algebras for which the first element of the pair is $n$-subhomogeneous, we see that $T$ is effectively axiomatizable (this follows from the effective axiomatization of $n$-subhomogeneous \cstar-algebras given in \cite[Subsection 2.5.4]{Muenster}), all models of $T$ are nuclear pairs (and thus Tsirelson pairs), and has at least one model that is a strong Tsirelson pair (by our contradiction assumption).  This contradicts the statement of Theorem \ref{noeffective}.
\end{remark}

\begin{cor}\label{effcor}
There can be no effectively axiomatizable theory $T$ of \cstar-algebras, all of whose models have the TP, and which has at least one model that has the STP.  In particular, neither the elementary class of \cstar-algebras with the TP nor the elementary class of \cstar-algebras with the STP are effectively axiomatizable.  
\end{cor}

The following consequence of the previous corollary improves upon \cite[Theorem 2.2]{AGH}.

\begin{cor}
There can be no effectively axiomatizable theory $T$ of \cstar-algebras, all of whose models have the QWEP property, and which has at least one model that is not subhomogeneous.
\end{cor}

\begin{proof}
This follows immediately from Corollaries \ref{QWEPTP} and \ref{effcor}.
\end{proof}

Using Corollary \ref{effcor}, an argument identical to that establishing \cite[Corollary 2.4]{AGH} establishes the following:

\begin{cor}
The class of \cstar-algebras without the (S)TP is not closed under ultraproducts.
\end{cor}

\section{Conflicts of interests statements}

The authors have no competing interests to declare that are relevant to the content of this article.

\end{document}